\documentclass[11pt]{amsart}

\usepackage{amssymb, amsmath}
\usepackage{diagbox}

\usepackage{tikz, tikz-cd}
\usepackage{enumitem}
\usepackage{mathtools}
\usepackage{todonotes}
\usepackage{amsthm}
\usepackage{amsfonts}
\usepackage{amsmath}
\usepackage[all]{xy}

\theoremstyle{plain}
 \newtheorem{thm}{Theorem}[section]

\theoremstyle{definition}

\theoremstyle{plain}
\newtheorem{theorem}[thm]{Theorem}

\newtheorem{lemma}[thm]{Lemma}

\theoremstyle{definition}

\newtheorem{defin}[thm]{Definition}

\newtheorem{remark}[thm]{Remark}

\numberwithin{equation}{section}

\newcommand{\hol}{\ensuremath{\mathcal{O}}}

\newcommand\la{\lambda}

\newcommand\s{\sigma}

\newcommand\e{\epsilon}

\newcommand\Ga{\Gamma}

\newcommand\ga{\gamma}
\newcommand\de{\delta}

\DeclareMathOperator{\Pic}{Pic}

\newcommand{\ra}{\ensuremath{\rightarrow}}

\newcommand{\RR}{\mathbb{R}}
\newcommand{\ZZ}{\mathbb{Z}}

\begin{document}
\title[
Irrational pencils and  VIP's]{Irrational pencils, and
 characterization of Varieties isogenous to a product,
via the Profinite completion of the  Fundamental group}
\author{by Fabrizio Catanese,}
\author{with Appendix by Pavel Zalesskii}
\date{\today}

\address{Fabrizio Catanese,
 Mathematisches Institut der Universit\"{a}t
Bayreuth, NW II\\ Universit\"{a}tsstr. 30,
95447 Bayreuth, Germany.}
\email{Fabrizio.Catanese@uni-bayreuth.de}
\address{Department of Mathematics, University of Brasília,
70910-900, Bras\'ilia-DF, Brazil}
\email{zalesskii@gmail.com}


\maketitle

\begin{abstract}
We give a very short proof of two  Theorems, whose content is outlined in the title, and where $\Pi_g$ is the fundamental group of
a compact complex curve of genus $g$:

(1) Theorem \ref{PIP} of the irrational pencil in the profinite version, 
saying that for a compact K\"ahler manifold an irrational pencil,
that is, a fibration onto a curve of genus $g \geq 2$, corresponds to a surjection of the profinite completion
$\widehat{\pi}_1(X) \ra \widehat{\Pi_g}$, which satisfies a maximality property;

(2)  Theorem \ref{PVIP} on the characterization of varieties isogenous to a product, 
profinite version, giving in particular a criterion for $X$
a compact K\"ahler manifold to be isomorphic to a product of curves of genera at least 2: if and only if 
$\widehat{\pi}_1(X) \cong \prod_1^n  \widehat{\Pi_{g_i}}$,
and some volume or cohomological condition is satisfied.

Theorem \ref{PVIP} yields a stronger result  than
the Main Theorem A of  \cite{authors}.

\end{abstract}

\section{Introduction}

\begin{defin}Let $X$ be a smooth projective variety.
$X$ is said to be a {\bf Variety isogenous to a  higher product}, or a VIP,  see \cite{isogenous},   if
$$W \cong (C_1 \times C_2 \times \dots C_n)/G,$$
where the $C_i$'s are  smooth curves of genus  $g(C_i)= : g_i \ge 2$ and $G$ is a finite group acting  freely on
the product variety $(C_1 \times C_2 \times \dots C_n)$.

\end{defin}

In \cite{isogenous} the following was proved in theorems 7.1 , 7.5, 7.7:

\begin{theorem}\label{isogenous-char}
Let $X$ be a compact K\"ahler manifold of dimension $n$ such that 

(i)  $\pi_1(X) \cong \Ga$ 
admits an index $d$ subgroup $\Ga'$ such that
$$\Ga'  
\cong \Pi_{g_1} \times \dots \Pi_{g_n } ,$$
where $\Pi_g$ is the fundamental group of a compact complex curve of genus $g$, and all $g_j \geq 2$, 

(ii) $H^{2n}(X, \ZZ)$ has index $d$ inside the image of $H^{2n}(\Ga', \ZZ) $ in $H^{2n}(X', \ZZ)$,
$X' \ra X$ being the unramified covering associated to $\Ga'$
(i.e.,  $H^{2n}(\Ga', \ZZ) = H^{2n}(X', \ZZ)$).

Then $X$ is the blow up of a VIP  $W$, and $X$ is   isogenous to a product ($X \cong W $)
if $K_X$ is ample.

Alternatively, $X$ is a VIP if 

(i), (ii bis) and (iii) hold true, where

(ii bis) the image of $H^{2n}(\Ga', \ZZ) $ inside  $H^{2n}(X', \ZZ)$ is nonzero, 

(iii) $K_X^n =  \frac{1}{d} 2^n n! \prod_j (g_j-1)$.

Alternatively, $X$ is a VIP if $K_X$ is ample, and (i), (ii') hold true,
where

(ii') $H^{2n}(X, \ZZ) \cong H^{2n}(\Ga, \ZZ)$.
\end{theorem}

\begin{remark}
Conditions (ii), (ii')  are verified  if $X$ is aspherical, that is, $X$ is a $ K(\Ga,1)$ space.
\end{remark} 

The heart of the proof boils down to treating the case where $\Ga' = \Ga$, 
showing that then $X$ is biholomorphic to a product of curves.

The aim of this short note is to give a short proof of the following result. It was  proven in  \cite{authors} as Theorem A using  assumption (I), and 
the conditions: $\pi_1(X)$ is residually finite,  and 
$X$ is aspherical.

\begin{theorem}\label{PVIP}
In Theorem \ref{isogenous-char} one obtains the same conclusions  replacing

(I)  condition (i) by condition (i-alg) $$\widehat{\Ga' } 
\cong \widehat{\Pi}_{g_1} \times \dots \widehat{\Pi}_{g_n } ,$$
where  $\widehat{\Ga}$ denotes the profinite completion of $\Ga$ ($\widehat{\Ga'} = \widehat{\pi_1(X')}$ may be called  
 the algebraic  fundamental group of $X'$ and be denoted $ \pi_1^{alg} (X') $),
 
 (II) adding  in conditions (ii), respectively (ii bis), (ii'), that $\Ga$ is residually finite.
\end{theorem}

It suffices then to prove Theorem \ref{PVIP} in the case where we assume that 
 $$ \widehat{\pi_1 (X) } \cong {\Pi}_{g_1} \times \dots \widehat{\Pi}_{g_n }.$$

 As in \cite{isogenous} the main point is to show that there exist holomorphic maps $f_i :  X\  \ra C_i$,
 where $C_i$ is a complex curve of genus $g_i$, which induce the 
 i-th homomorphism $ (f_i)_*\colon  \widehat{\pi_1 (X) }   \ra  \widehat{\pi_1 (C_i) } \cong \widehat{\Pi}_{g_i}$.
 
 This will be done in the next section, in Theorem \ref{PIP},
 and is based on earlier results about irrational pencils, contained
 in \cite{alb-gen}, \cite{gl}.
 
 We shall then conclude the proof of Theorem \ref{PVIP} in the final section.
 
 The arguments used concern properties of the fundamental
 groups $\Pi_g$,  of  their profinite completions,
 and especially non splitting properties of surjections between them.
 
 One such result is due to Pavel Zalesskii, and is proven in the
 appendix.

\section{Irrational pencils Theorem via  the algebraic fundamental group}

\begin{theorem}\label{PIP}
Let $X$ be a compact K\"ahler manifold of dimension $n$ and 
$$\psi :  \pi_1^{alg} (X): = \widehat{\pi_1(X) }  \ra  \widehat{\Pi}_{g}$$
be a surjective homomorphism, where $g \geq 2$.

Then there exists a holomorphic map $f: X \ra C$ with connected fibres, where $C$ is a complex curve of genus $ g_2 \geq g$,
such that  $\psi$ factors through the surjection $\widehat{\pi_1(f)}$. And $\psi = \widehat{\pi_1(f)}$ if the homomorphism
$\psi$  is maximal
with respect to such  factorizations  $ \widehat{\pi_1(X)}   \ra  \widehat{\Pi}_{g_2} \ra  \widehat{\Pi}_{g}$.\end{theorem}

\begin{remark}
The converse assertion clearly holds, such a fibration $f$ yields
a surjection  $\widehat{\pi_1(f)}$. 
\end{remark}

The shortest proof we can give is following the arguments of Beauville in the Appendix to 
\cite{alb-gen}.
For the sake of clarity we first discuss  some easy results concerning profinite completions.

\subsection{Profinite completion} Let $A$ be a group and let $\widehat{A}$ be its profinite completion,
that is,  as a set, the  inverse limit:
\begin{equation}\label{inverse}
 \widehat{A} = \varprojlim_{(H < _{fi} A)}  (A/H),
 \end{equation}
where $H <_{fi}$ means that $H$ is a  subgroup of finite index.

Since every subgroup $H$ of finite index contains a normal subgroup 
of finite index $H'$, then
$$ \widehat{A} = \varprojlim_{(H \lhd _{fi} A)}  (A/H),$$
where $\lhd _{fi} $ means that we have a normal subgroup of finite index.

Hence $ \widehat{A}$ is a group.

Given a finite  group $B$ and applying $Hom$ we get a direct limit
$$ Hom (\widehat{A} , B) = \varinjlim_{(H \lhd_{fi} A)}  Hom(A/H, B).$$

The above follows from the universal property of the profinite completion, from which essentially follows  also the following result.

\begin{lemma}\label{1}
Let $A,B$ be groups and assume that we have a surjection $\eta : \widehat{A} \twoheadrightarrow \widehat{B}$.

Then, for each finite quotient $ \phi : B \twoheadrightarrow G$, there is a surjection $\psi : A \twoheadrightarrow G$
such that  $ \widehat{\psi} = \widehat{\phi} \circ \eta$.

\end{lemma}
\begin{proof}
We define $\psi$ through the composition $ A \ra  \widehat{A} \ra \widehat{B}   \ra G$,
the last homomorphism being $\widehat{\phi}$.

$\psi$ factors through $ A \ra A/K$, $ K : = ker (\psi)$, and it suffices to show that $\psi$ is surjective.

$ A \ra G$ induces $  \widehat{\psi} :  \widehat{A}    \ra G$, which is surjective since $\eta : \widehat{A} \ra \widehat{B}$
and $ \widehat{\phi} :  \widehat{B} \ra G$ are both surjective.

Now, we have a homomorphism $ A/H \ra G$ iff $ H < K$, and, by the surjectivity of $  \widehat{\psi} $,
we can  find such a finite index normal subgroup such that $ A/H \ra G$ is surjective.

Since the composition  $ A/H \ra A/ K  \ra G$ is surjective, then $A/K \cong G$, 
and we are done.

\end{proof}

\begin{lemma}\label{2}

With the notation of Lemma \ref{1}, if $\widehat{A} \ra G$ is onto, then its kernel is isomorphic to $\widehat{K}$.
\end{lemma}
\begin{proof}
Assume that $H < K$:  then the cosets $K/H$ map to the identity
in $A/K = G$.  If moreover $H \lhd A$, then $A/H \twoheadrightarrow A/K = G$.

Then the elements in the Kernel come from $K/H$, with $H$ 
normal of finite index in $K$.

But every finite index subgroup $H$ of $K$ contains a finite index subgroup $H'$ which is normal in $A$.
\end{proof}

For the next lemma, assume that $M$ is a finitely generated abelian group $ M \cong \ZZ^b \oplus T$,
where $T$ is a finite abelian group.

\begin{defin}
For $p$ a prime number, let $M_p : = M / p M$  
and observe that, for $ p >> 0$, $M_p = (\ZZ/p)^b$.
\end{defin} 

\begin{lemma}\label{3}

 Let $A^{ab} : = A / [A,A]$ be the abelianization of $A$, and assume that
 $A^{ab} $ is finitely generated.
 
 Then  $\widehat{A} \ra A^{ab}_p$ is surjective.
\end{lemma}
\begin{proof}
Under our assumption $A^{ab}_p = A /K$, where $K$ is a finite index subgroup of $A$. 
\end{proof}

\begin{lemma}\label{4}

For $X$ a compact K\"ahler manifold, the irregularity $q(X): = h^1(\hol_X)$
can be so described:
$$q : = q(X) = max_{(p \ {\rm prime}, \ p >>0)} \{ r | \pi_1(X) \twoheadrightarrow (\ZZ/p)^{2r}\},$$
$$q : = q(X) = max_{(p \ {\rm prime}, \ p >>0)} \{ r | \widehat{\pi}_1(X) \twoheadrightarrow (\ZZ/p)^{2r}\}.$$
\end{lemma}

\begin{proof}
$\pi_1(X) \twoheadrightarrow H_1(X, \ZZ) = \ZZ^{2 q} \oplus T$, where $T$ is the Torsion subgroup,
which is finitely generated.

Hence, for $ p >>0$, $Hom (T, \ZZ/p)=0$.

Hence $ \pi_1(X) \twoheadrightarrow (\ZZ/p)^{2q}$, and every such surjection factors through it.

The second assertion follows from Lemma \ref{1}.
\end{proof}

\subsection{Proof of Theorem \ref{PIP}}

For each prime number $p, p>>0$, we have 
$$ \Pi_g \twoheadrightarrow 
\ZZ^{2g} \twoheadrightarrow (\ZZ/p)^{2g},$$
hence, by Lemma \ref{3}, $ \widehat{\Pi}_g \twoheadrightarrow  (\ZZ/p)^{2g},$
and by Lemma \ref{1} there is a surjection $\phi : \pi_1(X) \twoheadrightarrow  (\ZZ/p)^{2g}.$

Consider now the $\frac{p^{2g} -1}{p-1} = 1 + p + \dots p^{2g-1}$ kernels of epimorphisms 
$\e : (\ZZ/p)^{2g} \twoheadrightarrow \ZZ/p$.

It will suffice, as in the appendix to \cite{alb-gen}, to see that we get in this way at least 
$p^{2g-1}$ special elements of torsion order $p$ in $Pic^0(X)$.

`Special elements' means  that $ker (\eta)$, $\eta := \e \circ \phi$, yields a covering of $X$ with irregularity $ q' > q(X)$.

By Lemma \ref{2} we have that $\widehat{ker \eta}$ is the kernel of $\widehat{\pi}_1(X)  \twoheadrightarrow \ZZ/p$, 
and likewise by Lemma \ref{2} the kernel of $\widehat{\Pi}_g  \twoheadrightarrow \ZZ/p$
is isomorphic to $\widehat{\Pi}_{p (g-1) +1}$.

We observe that  we have a map $\Psi$ from $\widehat{ker \eta}$  to $\widehat{\Pi}_{p (g-1) +1}$, and since 
 by assumption we have a surjection $\widehat{\pi}_1(X)  \twoheadrightarrow \widehat{\Pi}_g$ , 
 it follows that $\Psi$ is surjective.
 
 By Lemma \ref{4}, 
 $$ q' : =   q(X') = max_{(p \ {\rm prime}, \ p >>0)} \{ r | \widehat{ker \eta} \twoheadrightarrow (\ZZ/p)^{2r}\}
 \geq $$
 
 $$ \geq max_{(p \ {\rm prime}, \ p >>0)} \{ r | \widehat{\Pi}_{p (g-1) +1} \twoheadrightarrow (\ZZ/p)^{2r}\} =
 2 (p(g-1) + 1)$$
 
 Hence  $$q' \geq p(g-1) + 1. $$
 
 Since $ g \geq 2$, for $ p >>0$, $q' > q(X)$ and Theorem   \ref{PIP} is proven.
 
 \section{Proof of Theorem \ref{PVIP}}
 
 As already observed, it  suffices  to prove Theorem \ref{PVIP} in the case where we assume that 
 \begin{equation}\label{product}
  \widehat{\pi_1(X) } \cong \widehat{\Pi}_{g_1} \times \dots \widehat{\Pi}_{g_n }.
  \end{equation}
 
 By Theorem \ref{PIP} to each  surjection $\psi_i : \widehat{\pi_1(X) }  \twoheadrightarrow \widehat{\Pi}_{g_i}$
 corresponds an irrational pencil, that is, a holomorphic map $f_i : X \ra C_i$ with connected fibres
 where the genus of $C_i$
 equals $g'_i \geq g_i$. Moreover, $\pi_1(f_i) : \pi_1(X)
 \twoheadrightarrow \Pi_{g'_i}$ is a surjection 
  inducing  a surjection  $\psi'_i : \widehat{\pi_1(X) }  \twoheadrightarrow \widehat{\Pi}_{g'_i}$ such that $\psi_i$ factors through $\psi'_i $.
 
 {\bf Step I: $ g'_i = g_i, \forall i.$}
 
 Assume that $ g'_i > g_i$. Then the projection on the i-th factor of $\widehat{\Pi}_{g_1} \times \dots \widehat{\Pi}_{g_n }$
 factors through $ \widehat{\Pi}_{g'_i}$, and  we would have   homomorphisms 
 $\widehat{\Pi}_{g_i} \ra \widehat{\Pi}_{g'_i}$, and surjections $\widehat{\Pi}_{g'_i} \twoheadrightarrow \widehat{\Pi}_{g_i}$ whose composition is the identity (we shall refer to this situation
 calling it a {\bf splitting surjection}).

  {\bf First proof of Step I :} such a splitting surjection
  cannot exist for $g'_i > g_i$, by the next Lemma \ref{splitting}, proven by Pavel Zalesskii in the appendix.
  
  {\bf Lemma 3.4}
For $g' > g$, there is no splitting surjection

$\widehat{\Pi}_{g'} \twoheadrightarrow \widehat{\Pi}_{g}$.
 
    {\bf Second  proof of Step I :} we shall give another, more elementary proof.
    
  Recall that Green and Lazarsfeld 
  \cite{gl} have proven that the cohomology jumping locus
  $$\{ L \in \Pic(X) | H^1(X,L) \neq 0\}$$ 
  is a finite union of translates of complex subtori of $\Pic(X)$,
  and each irrational pencil $f_i : X \ra C_i$ yields a subtorus $B_i$
  of dimension $g'_i$
  passing through the origin.
  
  By our assumption \eqref{product}, the tangent space to $\Pic(X)$ at the origin
  is a direct sum $\oplus_i V_i$ of complex subspaces of dimension $g_i$ 
  such that $V_i \subset  T_{B_i}$.
  
 Hence either all the tori $B_i$ are distinct, and $g'_i = g_i$,
 or each $B_j$ has tangent space 
 \begin{equation}\label{onto}
 T_{B_j} = \oplus_{i \ {\rm s.t. } \  V_i \subset  T_{B_j}} V_i
 \end{equation}
 at the origin.
 
 Correspondingly, we get a surjection
 
 $$\widehat{ \pi_1(X)}  \cong \widehat{\Pi}_{g_1} \times \dots \widehat{\Pi}_{g_n }
 \twoheadrightarrow \widehat{\Pi}_{g'_j},$$
 which yields for each $B_j$ a surjection 

\begin{equation}\label{split}
  \prod _{i \ {\rm s.t. } \  V_i \subset  T_{B_j}}  \widehat{\Pi}_{g_i}  \twoheadrightarrow \widehat{\Pi}_{g'_j},
  \end{equation}
such that each subproduct does not map surjectively.

Indeed, as we argued before, the projection on the i-th factor of $\widehat{\Pi}_{g_1} \times \dots \times \widehat{\Pi}_{g_n }$
 factors through $ \widehat{\Pi}_{g'_i}$,
 hence the surjection \eqref{split} is a splitting one.
 
 This contradicts part (b) of he following Lemma, hence Step I is established by proving 
 
 \begin{lemma}\label{sg}
 (a) the fundamental group $\Pi_g$ of a compact complex curve
 $C$ of genus $g \geq 2$ cannot be a  nontrivial semidirect
 product  $ K \rtimes H$.
 
 (b) There cannot be a split surjection 
 $$ \prod_1^k  \widehat{\Pi}_{g_i}  \twoheadrightarrow \widehat{\Pi}_{g'},$$
 with $ g'  = g_1 + \dots + g_k$.

 \end{lemma}

 \begin{proof}
 (a): since $\Pi_g$ contains no elements of finite order, then
 both subgroups $H,K$ are infinite.
 
 Hence the  covering spaces corresponding to the subgroups $K,H$ yield non compact curves.
 
 Therefore both $H,K$ are free groups, and a minimal presentation
 of $ K \rtimes H$ is
 $$ \langle \ga_1, \dots, \ga_m, \de_1, \dots, \de_n| 
 \de_i \ga_j \de_i^{-1} = M_{i,j}(\ga), \forall 1 \leq j \leq m,1 \leq i \leq n
 \rangle.$$
 It must be $ 2g = m +n$, and then we must have only one relation,
 hence $n=m=1$.
 
 In this case $g=2$ but since then $\de \ga \de^{-1} = \ga^{\pm1},$
 we get $\Pi_2 \cong \ZZ^2$, or $(\Pi_2)^{ab} \cong \ZZ \oplus \ZZ/2$, 
 a contradiction.
 
 Proof of (b).

 We have a
 splitting surjection 
 $$ \prod_1^k  \widehat{\Pi}_{g_i}  \twoheadrightarrow \widehat{\Pi}_{g'},$$
 with $ g'  = g_1 + \dots + g_k$, and setting $g : = g_1$,
 we can write it  as 
  $$\widehat{\Pi}_{g} \times    K'' \twoheadrightarrow \widehat{\Pi}_{g'},$$
  $K''$ being  a direct sum of
  profinite completions of surface groups.
  
   Let $K$ be the kernel of $\Pi_{g'} \ra \widehat{\Pi}_{g}$,
  and $H$ be its image. Let moreover 
  $H'$ be the kernel of $\Pi_{g'} \ra K''$.
  
  $H'$ is nonzero, else $\Pi_{g'} \ra K'$ is an embedding,
  contradicting the surjection of tangent spaces
  \eqref{onto}, similarly $K$ is nontrivial.
  
  By our assumptions, $H', K$ are commuting subgroups
  of $ \Pi_{g'} $.
  
  The next lemma is an old tool in complex function theory, which I learnt from  \cite{siegel}.
  
  \begin{lemma}\label{comm}
  If $H', K$ are nontrivial commuting subgroups
  of $ \Pi_{g'} $, then $H' \cdot  K$ is a commutative sugbroup.
  And a  normal  subgroup if $H', K$ are normal.

  \end{lemma}
  \begin{proof}
Given two M\"obius transformations of hyperbolic type,
different from the identity, and not of order 2,
assume that the first equals $ z \mapsto \mu z$, $\mu \in \RR^*$
$\mu \neq 1$, and the second
$$ \s: z \mapsto \frac{az  + b}{cz+d} .$$

Then  commutation of the two transformations  is equivalent to 
$$\mu (az+b) (c \mu z+d) \equiv  (cz + d) (a \mu z+b)  ,$$
 an equality of degree 2 polynomials which amounts to  
 $$ ac =0, \ bd=0.$$
 
  If $b=c=0$, then $\s (z) = \la z$,
if $a=d=0$, then $\s (z) = \la z^{-1}$,
and commutation holds iff $\mu \la = \la \mu^{-1}$,
which implies $\mu= -1$, a contradiction.

Hence the two transformations belong to a commutative subgroup,
which contains any transformation commuting with one of them.

 \end{proof}
 
 Now, $K \cap H' $ consists only of the identity,
 hence $K \cdot H'$ would be a commutative normal subgroup
 of rank at least 2.
 
  We reach a contradiction: because   $ \Pi_{g'} $, does not contain 
  such a normal commutative subgroup $G$, since every normal
  subgroup of infinite index is a free group.
  
  And, if $G$ were of finite index, then it would be a surface group of
  genus $\geq 2~$, hence not abelian.

     \end{proof}
     
     \begin{remark}
     Serge Cantat observed that Lemma  \ref{comm} goes  back to Fricke and Klein: it suffices to show that the two transformations have two common fixpoints, hence they are simultaneously diagonalizable.
     
     Each has an attractive fixpoint and a repelling fixpoint, and since they commute these pairs of points are the same.
     
     He also sketched  an alternative proof of the second statement of Lemma 
      \ref{sg} a):  the limit set of $K$ (in the boundary of the Poincar\'e disk) is invariant by $H$, hence is everything, therefore  $K$ must be
      of  finite index in $\Pi_g$.
     
     \end{remark}
 
  {\bf Step II}  We may now assume $g'_i = g_i, \ \forall i$.
  
 Putting together these maps, we obtain 
 $$ F : = (f_1 \times \dots \times f_n) : X \ra C_1 \times \dots \times C_n,$$
 where $F_* : = \pi_1(F)$ is surjective, because $\psi$ is surjective.
 
 
 Now, since $ \Ga$ is residually finite, it embeds into its completion, as well as $ \Pi_{g_1} \times \dots \times \Pi_{g_n}$.
 Hence $ F_* $ is also injective, hence an isomomorphism.

 The surjectivity of $F$ follows then by (ii), or (ii bis), or ( ii'). 
 
 Then the surjectivity of $F$ and  equality (iii), plus the fact that $K_X$ is ample, show that
 $F$ is an isomorphism.
  
 If we assume instead  the equality    $H^{2n} (\Ga, \ZZ) = H^{2n} (X, \ZZ)$, then $F$ is a blow   up, and an isomorphism 
 if $K_X$ is ample.
 
 \medskip
 
\qed

\subsection{Appendix by Pavel Zalesskii}

\begin{lemma}\label{splitting}
For $g' > g$, there is no splitting surjection
$\widehat{\Pi}_{g'} \twoheadrightarrow \widehat{\Pi}_{g}$.
 \end{lemma} 
  \begin{proof} One has a commutative diagram
  $$\xymatrix{\widehat{\Pi}_{g'}\ar[r]\ar[d]^{\pi_{g'}}& \widehat{\Pi}_{g}\ar[d]^{\pi_g}\\
  {\widehat{\Pi}^{(p)}}_{g'}\ar[r]& {\widehat{\Pi}^{(p)}}_{g}}$$ 
where   $^{(p)}$ means the pro-$p$ completion.

  Suppose on the contrary that we have such a splitting $f:\widehat{\Pi}_{g}\longrightarrow \widehat{\Pi}_{g'}$.  Since $\widehat{\Pi}_{g}$ is $2g$ -generated, so is $\widehat{\Pi}^{(p)}_{g}$ and $f(\widehat{\Pi}_{g})$. It follows that $\pi_{g'}(f(\widehat{\Pi}_{g}))$ is a 2g-generated pro-$p$ group and $\pi_{g'}f$ factors through $\pi_g$. Thus the lower horizontal map splits. Since $\widehat{\Pi}^{(p)}_{g'}$ is $2g'$-generated,  $\pi_{g'}(f(\widehat{\Pi}_{g}))$ has infinite index in $\widehat{\Pi}^{(p)}_{g'}$.  But every subgroup of  infinite index of
a  pro-$p$ surface group is free  pro-$p$  (this is true for Demushkin groups, see \cite[Exercise 5 on page 44]{S 65},  and a pro-$p$ surface group is a particular case of it), so $\pi_{g'}(f(\widehat{\Pi}_{g}))$ must be a free pro-$p$ group, contradicting the fact that
  $\pi_{g'}(f(\widehat{\Pi}_{g})) \cong {\widehat{\Pi}^{(p)}}_{g}$ has cohomological dimension 2. This finishes the proof.

  \end{proof}

  {\em Acknowledgement: Fabrizio would like to thank Serge Cantat and Dan Segal for useful comments.}

\end{document}